\documentclass[11pt,a4paper]{article}

\usepackage{amsmath,amssymb, amsfonts, amsthm}
\usepackage{nccmath}   
\usepackage{hyperref}
\usepackage{soul}       

\usepackage{enumerate}  
\usepackage[british]{babel}

\newcommand{\F}{\mathcal{F}}

\newcommand{\RR}{\mathbb{R}}

\newcommand{\tp}{^{\mathsf{T}}}
\newcommand{\mtp}{^{-\mathsf{T}}}

\newcommand{\ee}{\mathrm{e}}

\DeclareMathOperator*\esssup{\mathrm{ess\,sup}}
\DeclareMathOperator*\essran{\mathrm{ess\,ran}}

\DeclareMathOperator\Grad{\mathrm{grad}}

\DeclareMathOperator\Ran{\mathrm{Ran}}

\usepackage{color}

\newtheorem{theorem}{Theorem}[section]
\newtheorem{lemma}[theorem]{Lemma}

\newtheorem{remark}[theorem]{Remark}

\numberwithin{equation}{section}

\author{D. R. Michiel Renger and Stefanie Schindler}
\title{Gradient flows for bounded linear evolution equations}

\date{\today}

\begin{document}
\maketitle

\begin{abstract}
We study linear evolution equations in separable Hilbert spaces defined by a bounded linear operator. We answer the question which of these equations can be written as a gradient flow, namely those for which the operator is real diagonalisable. The proof is constructive, from which we also derive geodesic lambda-convexity.
\end{abstract}
	
\section{Introduction}

A great amount of literature is dedicated to proving that certain evolution equations can be written as a gradient flow, see for example~\cite{JKO1998,Maas2011,Mielke2013,CancesMatthesNabet2019,Bartha2012}. Such results are important since a gradient flow structure shows that the equation is consistent with thermodynamic principles~\cite{Ottinger2005}, it facilitates the use of additional tools to prove existence, stability and convergence to equilibrium~\cite{AGS2008}, and it opens up the possibility of using certain numerical schemes~\cite{CarilloMatthesWolfram2020}. 

In~\cite{Bartha2012}, it is proven that every ordinary differential equation on a finite-dimensional manifold is a gradient flow if it possesses a strict Lyapunov function. Here in this paper, we answer the question of which linear equations of the form
\begin{equation}
  \dot x(t) = A x(t)
\label{eq:linear evolution}
\end{equation}
in a possibly infinite-dimensional separable Hilbert space $H$ can be written as a gradient flow if the evolution operator $A:H\to H$ is linear and bounded.

\paragraph{Gradient systems.}%
In this section, we give a short introduction to gradient systems in Hilbert spaces, where we follow the notation of \cite{Mielke2013}. For a basic and more extensive introduction to this topic, we refer to \cite{Chill2010}.

Since our results will be based on spectral theory we restrict our analysis to separable Hilbert spaces; for brevity we identify $H^*$ with $H$.
Throughout this paper a \emph{gradient system} will be a triple $(H,K,\F)$ where $H$ is a separable Hilbert space, $K(x):H\to H$ is a positive definite symmetric operator and $\F:H\to\RR$. 
For ease of presentation we implicitly assume that $K$ is continuous G{\^a}teaux-differentiable and $\F$ is twice continuous G{\^a}teaux-differentiable. Such $K$ and $\F$ satisfying these properties will be called an \emph{Onsager operator} and a \emph{free energy}, respectively. The aim is to find, given operator $A$ and Hilbert space $H$, a gradient system $(H,K,\F)$ so that $Ax=-K(x)D\F(x)$ for all $x\in H$. For such a gradient system one can rewrite \eqref{eq:linear evolution} as $\dot x(t) = -K(x(t))D\F(x(t))=-\Grad_{x(t)}\F(x(t))$ on the manifold defined by the corresponding metric:
\begin{align}
  d(x_1,x_2)^2 := \inf_{\gamma\in\Gamma(x_1,x_2)} \left\lbrace \int_0^1\!\Big\lbrack \sup_{\xi\in H}\big\langle \xi,2\dot\gamma(s)-K(x)\xi\big\rangle_H \Big\rbrack\,\mathrm{d}s\right\rbrace,
\label{eq: def of d_K}
\end{align}
with $\Gamma(x_1,x_2)$ denoting the space of all curves $\gamma \in C^1([0,1],H)$ connecting $x_1$ and $x_2$.

In the special case where the Onsager operator $K(x) \equiv K$ is independent of $x$ and the driving functional is quadratic, i.e. $\F(x)=\frac{1}{2}\langle B(x-\pi),x-\pi\rangle_H$ for some symmetric bounded linear operator $B: H \to H$ and equilibrium point $\pi\in H$, we shall call $(H,K,\F)$ a \emph{canonical gradient system}. 

More generally, we also allow the quadratic form $\tfrac12\langle \xi, K(x)\xi\rangle$ to be replaced by a more general $\Psi^*(x,\xi)$ that is strictly convex in the second argument and satisfies $\inf_{\xi\in H}\Psi^*(x,\xi)=\Psi^*(x,0)=0$. Such $\Psi^*$ is called a \emph{dissipation potential}, and we implicitly assume that $\Psi^*$ is continuous G{\^a}teaux-differentiable in the first argument and twice in the second. We then call $(H,\Psi^*,\F)$ a \emph{generalised gradient system} if $\Psi^*:H\times H\to\RR$ is a dissipation potential and $\F:H\to\RR$ is a free energy. Similar as before, we shall assume that besides $\Psi^*$ also $\F$ is twice continuous G{\^a}teaux-differentiable. The evolution equation corresponding to the flow of $(H,\Psi^*,\F)$ is $\dot x(t)=D_{\xi}\Psi^*\big(x(t),-D\F(x(t))\big)$, thus giving rise to a possible nonlinear relation between `driving forces' and `velocities'. These structures go back at least to~\cite{DeDonder1936}, and arise in a large variety of physical and mathematical problems, see \cite{MielkePeletierRenger2014,PRST2020} and the references therein.

There are good reasons to believe that symmetry of $A$ (i.e. reversibility of generators in the setting of Markov processes) is a sufficient condition for the evolution~\eqref{eq:linear evolution} to be a gradient flow equation: the foundations of this argument trace back to Boltzmann's H-Theorem~\cite{Boltzmann1872} and Onsager's reciprocity relations~\cite{Onsager1931I}, but can also be motivated in relation to large-deviation theory~\cite{MielkePeletierRenger2014}. Perhaps surprisingly, it turns out that the class of equations that are the flow of a gradient system is much bigger, namely those equations for which $A$ is real diagonalisable. 
We illustrate this result in the following paragraph.

\paragraph{Finite-state Markov chains.}  Consider the Kolmogorov equation of a finite-state continuous-time {M}arkov chain; in that case $A\in\RR^{d\times d}$ is a (transposed) generator matrix such that the evolution~\eqref{eq:linear evolution} preserves non-negativity and total probability. The discovery of \cite{Maas2011} and \cite{Mielke2013} showed that reversible Markov chains are the flow of a system driven by the relative entropy with respect to the invariant measure. For example if we define
\begin{align*}
  A&:=
  \begin{bmatrix}
    -2 & 1 & 1\\
    1 & -2 & 1\\
    1 & 1 & -2
  \end{bmatrix},
  \hspace{3cm}
  \F(x):=\sum_{i=1}^3 x_i\log 3x_i,\\
  K(x)&:=\frac13
  \begin{bmatrix}
    \sum_{j\neq 1} \Lambda(3x_1,3x_j) & -\Lambda(3x_1,3x_2) & -\Lambda(3x_1,3x_3)\\
    -\Lambda(3x_1,3x_2) & \sum_{j\neq 2} \Lambda(3x_2,3x_j) & -\Lambda(3x_1,3x_3)\\
    -\Lambda(3x_1,3x_3) & -\Lambda(3x_2,3x_2) & \sum_{j\neq 3} \Lambda(3x_3,3x_j)
  \end{bmatrix},
\end{align*}
where $\Lambda(a,b):=(a-b)/(\log a-\log b)$ is the logarithmic mean, then $K(x)$ is positive definite and symmetric, and indeed $Ax=-K(x)D\F(x)$.

However, \cite{Dietert2015} showed that reversibility is only a necessary assumption \emph{if} the driving functional $\F$ is the relative entropy; when allowing more general functionals, then the existence of a gradient system whose flow is the Kolmogorov forward equation~\eqref{eq:linear evolution} is equivalent to real diagonalisability of $A$. As an illustration, the Markov chain with the following transposed generator is not reversible but real diagonalisable, and it is the flow of the following (canonical) gradient system:
\begin{align*}
  A&:= 
    \begin{bmatrix}
      -2 & 0 & 2\\
      1 & -3 & 2\\
      1 & 3 & -4
    \end{bmatrix},
  \hspace{3cm}
  \F(x):= 
    \frac16 x\tp
    \begin{bmatrix}
      4 & -4 & 0 \\
     -4 &  6 & -2\\
      0 & -2 & 2
    \end{bmatrix}
    x,\\
  K(x)&:\equiv 
    \begin{bmatrix}
      3 & 3/2 & -3/2\\
      3/2 & 9/4 & -3/4\\
      -3/2 & -3/4 & 21/4
    \end{bmatrix}.
\end{align*}

\paragraph{Bounded linear operators on a separable Hilbert space.}
In this work we generalise the result of~\cite{Dietert2015} to linear equations~\eqref{eq:linear evolution} on a separable Hilbert space, not necessarily conserving non-negativity and total mass. We further generalise by allowing generalised gradient structures as well. Apart from assuming boundedness of the operator $A$ we also require surjectivity. The surjectivity is motivated by the fact that solutions to the linear equation~\eqref{eq:linear evolution} are contained in $\Ran(A)$ modulo the initial condition. Hence if $A$ is not surjective we can always restrict it -- without too much loss of generality -- to the Hilbert space $\Ran(A)$.  

Our first main result is the following:

\begin{theorem} Let $H$ be a separable Hilbert space, and $A:H\to H$ a surjective bounded linear operator. Then the following statements are equivalent:
\begin{enumerate}[(i)]
\item\label{th it: ggs} $Ax=D_\xi\Psi^*\big(x,-D\F(x)\big)$ for some generalised gradient system $(H,\Psi^*,\F)$ with some equilibrium $\pi\in H$, i.e. $D\F(\pi)=0$,
\item\label{th it: gs} $Ax=-K(x)D\F(x)$ for some gradient system $(H,K,\F)$ with some equilibrium $\pi\in H$,
\item\label{th it: can gs} $Ax= - KD\F(x)$ for some canonical gradient system $(H,K,\F)$ with some equilibrium $\pi\in H$,
\item\label{th it: real diag} $A$ is real diagonalisable, i.e. there exist:
\begin{enumerate}[(1)]
\item a measure space $(\Omega, \mathcal{A}, \mu)$,
\item an essentially bounded, real-valued $f\in L^\infty_\mu(\Omega)$ with corresponding bounded linear multiplication operator $M_f:L^2_{\mu}(\Omega)\to L^2_{\mu}(\Omega)$ defined by
\begin{align*}
  M_f(\varphi)(\omega) := \varphi(\omega) f(\omega)  &&\text{for }\mu\mathrm{-a.e.}\; \omega \in \Omega,
\end{align*}
\item an invertible bounded linear $V:H\to L^2_{\mu}(\Omega)$ so that \footnote{Note that $V$ does not need to be unitary, as often required in the literature~\cite[Sec.~3.6]{Kubrusly2012}.}
\begin{align*}
  A =V^{-1} M_f V.
\end{align*}
\end{enumerate}
\end{enumerate}
\label{th:main result}
\end{theorem}

We postpone the proof of this result to Section~\ref{sec:gs and diag}. Let us mention here that the proof of \emph{(iv)} $\implies$ \emph{(iii)} will be based on the explicit construction
\begin{align}
  \F(x):=-\mfrac12 (Vx,M_fVx)_{L^2_\mu(\Omega)} &&\text{and}&&  K:= V^{-1}V\mtp,
\label{eq:construction FK}
\end{align}
where $V\tp: L^2_{\mu}(\Omega) \to H$ is the adjoint of $V$, and $V\mtp$ is the adjoint of the inverse or equivalently the inverse of the adjoint. 

Recall from the spectral theorem that any symmetric bounded operator $A$ is real diagonalisable. Hence symmetry is indeed a sufficient condition for the evolution equation~\eqref{eq:linear evolution} to be a gradient flow, which is consistent with Onsager theory as mentioned above. However, our Theorem~\ref{th:main result} shows that the necessary condition is not symmetry but real diagonalisability.

As a byproduct of Theorem~\ref{th:main result}, we derive geodesic $\lambda$-convexity of the constructed gradient structure~\eqref{eq:construction FK}.
The relevance of this property lies in the fact -- among other facts -- that it implicates a Lipschitz continuous dependence of the solutions on their initial data by the inequality, using the metric~\eqref{eq: def of d_K}, see for example ~\cite[Th.~4.0.4]{AGS2008}, \cite[Th.~2.6]{daneri2010lecture} 
\begin{equation*}
  d(x_1(t),x_2(t)) \leq \ee^{-\lambda t} d(x_1(0),x_2(0)), \quad  \text{for all } t \geq 0. 
\end{equation*}
When $\lambda \geq 0$, this yields a contraction semigroup, and in the special case $\lambda >0$, one even gains exponential decay towards the unique equilibrium state $\pi$. This provides a new perspective to the classical argument that convergence speed to equilibrium is related to the spectral gap of the operator. 

\begin{theorem}
\label{th:geodesic convexity}
Let $H$ be a separable Hilbert space and $A:H\to H$ a surjective, bounded, linear and real diagonalisable operator $A = V^{-1} M_f V$,
and let $(H, K,\F)$ be the canonical gradient system~\eqref{eq:construction FK}. Then $\F$ is geodesically $\lambda$-convex w.r.t. $d$ with
\begin{equation*}
  \lambda:= - \esssup_{\omega \in \Omega} f(\omega) \, c_V,
\end{equation*}
where the constant $c_V$ is given by 
\begin{align*}
 c_V := \left\{\begin{array}{ll}  \Vert V^{-1} \Vert_{L( L^2_\mu(\Omega),H)}^2 \Vert V \Vert_{L(H,  L^2_\mu(\Omega))}^2, & \text{if } \esssup_{\omega \in \Omega} f(\omega) \geq 0,  \\
        \Vert V^{-1} \Vert_{L( L^2_\mu(\Omega),H)}^{-2} \Vert V \Vert_{L(H,  L^2_\mu(\Omega))}^{-2}, & \text{else.} \end{array}\right.
\end{align*}

If we additionally assume that the measure space $(\Omega,\mathcal{A},\mu)$ associated to the diagonalisability of $A$ is finite and the spectrum satisfies $\sigma(A) \subset (-\infty, 0]$  
then $f\leq0$ $\mu$-a.e. and so $\lambda \geq 0$.
\end{theorem}

We give a proof of this theorem in Section~\ref{sec: geodesic conv}. After the two proof sections we conclude with a discussion in Section~\ref{sec:discussion}.

\vspace{1cm}

\section{Gradient systems and diagonalisability}
\label{sec:gs and diag}

This section is devoted our main Theorem~\ref{th:main result}, which is proven in three separate lemmas.

\begin{lemma} 
Statements \eqref{th it: ggs}, \eqref{th it: gs} and \eqref{th it: can gs} from Theorem~\ref{th:main result} are equivalent.
\end{lemma}

\begin{proof}
Clearly a canonical gradient system $(H,K,\F)$ is also a gradient system, which in turn is a generalised gradient system if we define $\Psi^*(x,\xi):=\tfrac12 \langle \xi,K\xi\rangle$. Hence, we only need to show \eqref{th it: ggs} $\implies$ \eqref{th it: can gs}.

Differentiating $Ax=D_\xi\Psi^*(x,-D\F(x))$ in the equilibrium point $\pi \in H$ and using that $A$ is linear yield:
\begin{align*}
  A &= D_x D_\xi \Psi^*\big(\pi,-D\F(\pi)\big) - D^2\F(\pi) D^2_\xi\Psi^*\big(\pi,-D\F(\pi)\big)\\
    &= -D^2_\xi\Psi^*\big(\pi,0\big)D^2\F(\pi),
\end{align*}
since $D\F(\pi)=0$ and $D_\xi\Psi^*(x,0)\equiv0$. Define the canonical gradient system by
\begin{align*}
  \hat\F(x):=\mfrac12\langle D^2\F(\pi) x,x\rangle_H &&\text{and}&& \hat K:=D^2_\xi\Psi^*(\pi,0),
\end{align*}
the latter being positive definite since $\Psi^*$ is assumed to be strictly convex. It follows that for all $x\in H$,
\begin{align*}
  Ax= -D^2_\xi\Psi^*\big(\pi,0\big)D^2\F(\pi)x= -\hat K D\hat F(x).
\end{align*}
\end{proof}

\begin{lemma}
Statement \eqref{th it: real diag} implies \eqref{th it: can gs} in Theorem~\ref{th:main result}.
\label{lem:diag implies gs}
\end{lemma}

\begin{proof}
Let $A$ be bounded, linear and real diagonalisable, see Theorem~\ref{th:main result}, and define $\F$ and $K$ by \eqref{eq:construction FK}.
It is easily seen that $(H, K, \F)$ is a canonical gradient system. Indeed, $ K$ is independent of $x$ and therefore smooth, it is symmetric due to $ K\tp = {\big(V^{-1}V\mtp\big)}\tp =  K$, and positive definite because:
\begin{align*}
  \langle \xi,  K\xi \rangle_H = \langle \xi, V^{-1}V\mtp \xi \rangle_H =  (V\mtp \xi, V\mtp \xi )_{L^2_\mu(\Omega)} > 0 &&\text{for all } 0\neq \xi\in H.
\end{align*}

Further, $\F$ is quadratic since $M_f$ is self-adjoint whenever $f$ is real-valued.
In particular we get $D \F(x) = -V\mtp M_f V x \in H$. Now we have
\begin{align*}
Ax = V^{-1} M_f V = V^{-1} V\mtp V\tp M_f V x = - K D  \F (x),
\end{align*}
which was to be demonstrated.  
\end{proof}

\begin{lemma}
Statement \eqref{th it: can gs} implies \eqref{th it: real diag} in Theorem~\ref{th:main result}.
\end{lemma}

\begin{proof}
  We are given a canonical gradient system $(H,K,\F)$ with $\F(x)=\frac12 \langle Bx,x\rangle_H$ so that $Ax=- K D\F(x)=- K Bx$ for all $x\in H$. Since $ K$ is symmetric positive definite it has a unique symmetric positive definite bounded square root $\sqrt{K}:H\to H$ such that $K=\sqrt{K}\tp\sqrt{K}=\sqrt{K}\sqrt{K}$ \cite[VII.104, Th.~4]{RieszNagy1953}.

We show that this linear operator $\sqrt{K}$ is invertible. For the injectivity, pick an arbitrary $x\in H$ for which $\sqrt{K}x=0$. Then
\begin{align*}
  \langle K x,x\rangle_H= \Big\langle \sqrt{K} x,\sqrt{K} x\Big\rangle_H = 0,
\end{align*}
and so $x=0$ by the positive definiteness of $K$. For the surjectivity, note that $A=-\sqrt{K}\sqrt{K} B$, and so by the assumed surjectivity of $A$ we have $H=\Ran(A)\subset\Ran(\sqrt{K})\subset H$.

Define the following two bounded linear operators:
\begin{align*}
  \bar A:=\sqrt{K}^{-1} A \sqrt{K}
  &&\text{and}&&
  \bar B:=\sqrt{K} B \sqrt{K}.
\end{align*}
Clearly, $\bar B$ is symmetric since $B$ is symmetric. From $A=-KB$ it follows that $\bar A=-\bar B$, and so $\bar A$ is also symmetric.

We now invoke the spectral theorem for symmetric bounded linear operators \cite[Th.~3.11]{Kubrusly2012} and obtain the existence of a finite measure space $(\Omega, \mathcal{A}, \mu)$, a unitary operator $U:H\to L^2_\mu(\Omega)$ and a real-valued $f\in L^\infty_\mu(\Omega)$ with corresponding multiplication operator $M_f$ such that $\bar A=U\tp M_f U$. After defining the invertible bounded linear operator $V:=U\sqrt{K}^{-1}:H\to L^2_\mu(\Omega)$, we conclude that
\begin{align*}
  A= \sqrt{K} \bar A \sqrt{K}^{-1}= V^{-1} M_f V.
\end{align*}
\end{proof}

\begin{remark}
Our arguments imply that $A$ is complex diagonalisable if and only if the linear equation can be written as a ``complex-valued gradient flow''. It is unclear to us whether such gradient structures have any practical relevance.
\end{remark}

\section{Geodesic $\lambda$-convexity}
\label{sec: geodesic conv}

The aim of this section is to prove Theorem \ref{th:geodesic convexity}: the geodesic $\lambda$-convexity of the function $\F$ with respect to the metric $d $. Throughout this section we assume that $A$ is surjective, bounded, linear and real diagonalisable, and $(H,K,\F)$ is the corresponding canonical gradient system~\eqref{eq:construction FK}. It follows from the construction that the inverse $G:=K^{-1}:H\to H$ exists, and so  
\begin{align}
  d (x_1,x_2)^2 = \inf_{\gamma\in\Gamma(x_1,x_2)} \left\lbrace \int_0^1\! \big\langle G\dot\gamma(s),\dot\gamma(s)\big\rangle_H\,\mathrm{d}s\right\rbrace.
\label{eq: def2 of d_K}
\end{align}
Recall~\cite{daneri2010lecture} that geodesic $\lambda$-convexity means that the mapping $s \mapsto \F(\gamma(s))$ is $\lambda$-convex for all arc-length parametrised (constant speed) geodesics $\gamma:[0,1] \to H$ , i.e. for all $\theta \in [0,1]$ we have
\begin{equation*}
  \F(\gamma(\theta)) \leq (1-\theta)\F(\gamma(0)) + \theta\F(\gamma(1))- \lambda\frac{\theta(1-\theta)}{2}d (\gamma(0),\gamma(1))^2.
\end{equation*}

Before we prove this inequality, we show the following lemma. For brevity we shall omit the indices from the norms
$\lVert\cdot\rVert_{L( L^2_\mu(\Omega),H)}$ and $\lVert\cdot\rVert_{L(H,  L^2_\mu(\Omega))}$.

\begin{lemma} The free energy functional $\F$ from~\eqref{eq:construction FK} is $\tilde\lambda$-convex in the Hilbert space $H$, where
\begin{align*}
  \tilde\lambda := -\esssup_{\omega \in \Omega} f(\omega) \, \tilde{c_V},
  &&\text{and}&&
  \tilde{c_V}:=\begin{cases}
    \Vert V \Vert^2,          &\text{if } \esssup_{\omega \in \Omega} f(\omega) \geq 0,\\
    \Vert V^{-1} \Vert^{-2},  &\text{otherwise.}
  \end{cases}
\end{align*}
\end{lemma}

\begin{proof}
It holds
\begin{align*}
\langle & D\F(x_1)- D\F(x_2), x_1-x_2\rangle_H =  - \langle  V^T M_f V(x_1 -x_2)  , x_1-x_2\rangle_ H \\
&= - \big(  M_f V(x_1 -x_2)  , V(x_1-x_2) \big)_{L^2_\mu(\Omega)} = - \int_{\Omega} f(\omega) \big(V(x-y)(\omega) \big)^2 d\mu(\omega) \\
& \geq -\esssup_{\omega \in \Omega} f(\omega) \Vert V(x_1-x_2) \Vert_{L^2_\mu(\Omega)}^2 \geq -\esssup_{\omega \in \Omega} f(\omega) \, \tilde{c_V} \, \Vert x_1-x_2 \Vert_H^2,
\end{align*}
which is equivalent to $\tilde\lambda$-convexity since $\F$ is G\^ateaux-differentiable (see \cite[Prop. 3.12]{Peypouquet2015}).
\end{proof}

Next, we show that in our case geodesic $\lambda$-convexity can be obtained from $\tilde\lambda$-convexity in $H$. More precisely, we have the following lemma:
\begin{lemma}
Let $d $ be the metric~\eqref{eq: def2 of d_K} with metric tensor $G(x):=K^{-1}(x)= V\tp V$, and let $\F:H \to \RR$ be a $\tilde\lambda$-convex functional in $H$. Then $\F$ is geodesically $\lambda$-convex with respect to the metric $d $, where
\begin{equation*}
\lambda := \left\{\begin{array}{ll} \tilde\lambda \Vert V \Vert^{-2}, & \text{if $\tilde\lambda > 0$,}  \\
          \tilde\lambda \Vert V^{-1} \Vert^2, & \text{if $\tilde\lambda \leq 0$.} \end{array}\right.
\end{equation*}
\end{lemma}

\begin{proof}
Let $x_1,x_2 \in H$ be arbitrary and let $\gamma \in \Gamma(x_1,x_2)$ be a geodesic connecting $x_1$ and $x_2$. As $K^{-1}$ is independent of $x$, the metric space is flat, and the geodesic between $x_1$ and $x_2$ is a straight line $\gamma(s) = (1-s)x_1 +s x_2$. Therefore
\begin{align*}
  d (x_1,x_2)^2 &\stackrel{\eqref{eq: def2 of d_K}}{=}  \int_0^1\! \big\langle V\tp V \dot\gamma(s),\dot\gamma(s)\big\rangle_H\,\mathrm{d}s
=  \Vert V(x_2-x_1) \Vert^2_{L^2_\mu}.
\end{align*}
This yields the estimates
\begin{equation}
  \frac1{\Vert V^{-1} \Vert} \Vert x_2-x_1 \Vert_H \leq d (x_1,x_2) \leq \Vert V \Vert \Vert x_2-x_1 \Vert_H.
\label{eq:metric bounds}
\end{equation}

First consider the case $\tilde\lambda \leq 0$. Exploiting the $\tilde\lambda$-convexity of $\F$, we can bound for arbitrary $\theta \in [0,1]$,
\begin{align*}
  \F(\gamma(\theta)) &= \F((1-\theta) x_1+\theta x_2) \\
    &\leq  (1-\theta)\F(x_1)+ \theta \F(x_2) - \frac{\tilde\lambda (1- \theta)\theta }{2} \Vert x_2 - x_1 \Vert_H^2\\
    &\!\!\!\stackrel{\eqref{eq:metric bounds}}{\leq}  (1-\theta)\F(x_1)+ \theta \F(x_2) - \frac{\tilde\lambda  \Vert V^{-1} \Vert^2 (1- \theta)\theta }{2}  d (x_1,x_2)^2 \\
    &=  (1-\theta)\F(\gamma(0))+ \theta \F(\gamma(1)) - \frac{\lambda (1- \theta)\theta }{2}  d (\gamma(0),\gamma(1))^2.
\end{align*}
Similarly for $\tilde\lambda > 0$,
\begin{align*}
  \F(\gamma(\theta)) 
    &\!\!\!\stackrel{\eqref{eq:metric bounds}}{\leq}  (1-\theta)\F(x_1)+ \theta \F(x_2) - \frac{\tilde\lambda (1- \theta)\theta }{2\Vert V \Vert^{2}}  d (x_1,x_2)^2 \\
    &=  (1-\theta)\F(\gamma(0))+ \theta \F(\gamma(1)) - \frac{\lambda (1- \theta)\theta }{2}  d (\gamma(0),\gamma(1))^2.
\end{align*}
\end{proof}

Putting the two lemmas together, we proved the first part of Theorem \ref{th:geodesic convexity}, the geodesically $\lambda$-convexity of $\F$ in the metric space $(H,d )$. The last claim is stated in the following lemma:

\begin{lemma}
Under the additional assumption that the measure space $(\Omega,\mathcal{A},\mu)$ associated to the diagonalisability of $A$ is finite and the spectrum satisfies $\sigma(A) \subset (-\infty, 0]$  
then $f\leq0$ $\mu$-a.e..
\label{lem:nonneg spectrum}
\end{lemma}

\begin{proof}
First, as the measure space $(\Omega, \mathcal{A}, \mu)$ is finite, the spectrum of $M_f$ is given by the essential range of $f$ (see \cite[Lemma 4.55]{Knapp2008} for a proof). Second, we use that similar operators have the same spectrum, and obtain
\begin{equation*}
\essran(f) = \sigma(M_f) = \sigma(V^{-1} A V) = \sigma(A) \subset  (-\infty, 0].
\end{equation*}
\end{proof}

\section{Discussion}
\label{sec:discussion}

We end our paper with a discussion about our assumptions and results.

First, our work is restricted to separable Hilbert spaces $H$ because our arguments are based on spectral theory. Further, the assumed boundedness of the operator $A$ can be generalised to a large extent. However, it turns out that for unbounded operators the free energy $\F$ generally fails to be G\^ateaux differentiable. Our assumptions that both $\F$ and $\Psi^*(x,\cdot)$ need to be twice G\^ateaux differentiable can probably be generalised by approximation arguments, but the once differentiability of $\F$ is fundamental in our definition of a gradient flow. This hints that possible workarounds to allow for unbounded operators may be found by using more general definitions of gradient flows that do not require differentiability, like ``curves of maximal slope'' and ``minimising movements'' \cite{AGS2008}. This is beyond the scope of the current paper.

Let us also mention that our restriction to linear equations is essential, again because we couple gradient systems to spectral theory, but also because we use a linearisation around the equilibrium to simplify any (generalised) gradient system to a canonical gradient system. Therefore we expect our results to be approximately true for nonlinear equations near the equilibrium state.  

Finally, the gradient system that we construct for real diagonalisable operators could be perceived as rather unnatural. After all, canonical gradient systems have quadratic free energy and do not allow for entropic expressions. In this regard, we stress that our main Theorem~\ref{th:main result} is an existence result without any claim of uniqueness.

\section*{Acknowledgements}

The research was partially supported by Deutsche Forschungsgemeinschaft (DFG)
via the Collaborative Research Center SFB\,910 ``Control of self-organizing
nonlinear systems'' (project number 163436311), subproject A5 ``Pattern
formation in coupled parabolic systems'' and via the Collaborative Research Center SFB\,1114 ``Scaling Cascades in Complex Systems'', project C08.
We thank Marcus Kaiser for the useful discussions.

\bibliographystyle{alpha}
\bibliography{library} 

\end{document}